\documentclass[12pt,reqno]{amsart}
\usepackage{amssymb, latexsym}
\usepackage{url}

\begin{document}
 \bibliographystyle{plain}

 \newtheorem{theorem}{Theorem}
 \newtheorem{lemma}[theorem]{Lemma}
 \newtheorem{corollary}[theorem]{Corollary}
 \newtheorem{problem}[theorem]{Problem}
 \newtheorem{conjecture}[theorem]{Conjecture}
 \newtheorem{definition}[theorem]{Definition}
 \newtheorem{prop}[theorem]{Proposition}
 \numberwithin{equation}{section}
 \numberwithin{theorem}{section}

 \newcommand{\mc}{\mathcal}
 \newcommand{\rar}{\rightarrow}
 \newcommand{\Rar}{\Rightarrow}
 \newcommand{\lar}{\leftarrow}
 \newcommand{\lrar}{\leftrightarrow}
 \newcommand{\Lrar}{\Leftrightarrow}
 \newcommand{\zpz}{\mathbb{Z}/p\mathbb{Z}}
 \newcommand{\mbb}{\mathbb}
 \newcommand{\B}{\mc{B}}
 \newcommand{\cc}{\mc{C}}
 \newcommand{\D}{\mc{D}}
 \newcommand{\E}{\mc{E}}
 \newcommand{\F}{\mc{F}}
 \newcommand{\G}{\mc{G}}
  \newcommand{\ZG}{\Z (G)}
 \newcommand{\FN}{\F_n}
 \newcommand{\I}{\mc{I}}
 \newcommand{\J}{\mc{J}}
 \newcommand{\M}{\mc{M}}
 \newcommand{\nn}{\mc{N}}
 \newcommand{\qq}{\mc{Q}}
 \newcommand{\PP}{\mc{P}}
 \newcommand{\U}{\mc{U}}
 \newcommand{\X}{\mc{X}}
 \newcommand{\Y}{\mc{Y}}
 \newcommand{\itQ}{\mc{Q}}
 \newcommand{\C}{\mathbb{C}}
 \newcommand{\R}{\mathbb{R}}
 \newcommand{\N}{\mathbb{N}}
 \newcommand{\Q}{\mathbb{Q}}
 \newcommand{\Z}{\mathbb{Z}}
 \newcommand{\A}{\mathbb{A}}
 \newcommand{\ff}{\mathfrak F}
 \newcommand{\fb}{f_{\beta}}
 \newcommand{\fg}{f_{\gamma}}
 \newcommand{\gb}{g_{\beta}}
 \newcommand{\vphi}{\varphi}
 \newcommand{\whXq}{\widehat{X}_q(0)}
 \newcommand{\Xnn}{g_{n,N}}
 \newcommand{\lf}{\left\lfloor}
 \newcommand{\rf}{\right\rfloor}
 \newcommand{\lQx}{L_Q(x)}
 \newcommand{\lQQ}{\frac{\lQx}{Q}}
 \newcommand{\rQx}{R_Q(x)}
 \newcommand{\rQQ}{\frac{\rQx}{Q}}
 \newcommand{\elQ}{\ell_Q(\alpha )}
 \newcommand{\oa}{\overline{a}}
 \newcommand{\oI}{\overline{I}}
 \newcommand{\dx}{\text{\rm d}x}
 \newcommand{\dy}{\text{\rm d}y}
\newcommand{\cal}[1]{\mathcal{#1}}
\newcommand{\cH}{{\cal H}}
\newcommand{\diam}{\operatorname{diam}}

\parskip=0.5ex

\title[Density of orbits in the Adeles]{Density of orbits of semigroups of\\endomorphisms acting on the Adeles}
\author{Alan~Haynes and Sara~Munday}
\thanks{Research of both authors supported by EPSRC grant EP/J00149X/1.}
\address{School of Mathematics, University of Bristol, Bristol, UK}
\email{alan.haynes@bristol.ac.uk, sara.munday@bris.ac.uk}

\begin{abstract}
We investigate the question of whether or not the orbit of a point in $\A/\Q$, under the natural action of a subset $\Sigma\subseteq\Q$, is dense in $\A/\Q$. We prove that if the set $\Sigma$ is a multiplicative semigroup of $\Q^\times$ which contains at least two multiplicatively independent elements, one of which is an integer, then the orbit $\Sigma\alpha$ of any point $\alpha$ with irrational real coordinate is dense.
\end{abstract}

\allowdisplaybreaks

\maketitle

\section{Introduction}
Let $S$ be a multiplicative semigroup of positive integers which contains at least two multiplicatively independent elements (i.e., elements whose logarithms are $\Q$-linearly independent). A well known theorem of Furstenberg \cite[Theorem IV.1]{Furs1967} states that if $\alpha\in\R/\Z$ is irrational then the set
\[S\alpha:=\{s\alpha : s\in S\}\subseteq\R/\Z\]
is dense in $\R/\Z$. This is a fundamental example of rigidity in dynamical systems, and it was extended by Berend in \cite{Bere1984} to study the action of $\Q$ on the Adelic quotient $\A/\Q$. Berend proved that under certain conditions on a multiplicative semigroup $\Sigma\subseteq\Q$, every point $\alpha\in\A/\Q$ with an infinite orbit under the action of $\Sigma$ turns out to have a dense orbit. Berend's hypotheses on $\Sigma$ require it to contain elements whose absolute values are greater than one, in all completions of $\Q$ (Archimedean and non-Archimedean).

In this paper we study semigroups of $\Q$ which are generated by only two elements. In this case the dichotomy of finite and dense orbits is no longer present (see Section \ref{sec.non dense}). Our main result is the following theorem.
\begin{theorem}\label{thm.dense orbits}
Suppose that $\gamma\in\Q$ and $\delta\in\Z$ are multiplicatively independent and let
\[\Sigma=\Sigma(\gamma,\delta):=\{\gamma^a\delta^b:a,b\in\N\}.\]
Then for any $\alpha\in\A/\Q$ with $\alpha_\infty\not\in\Q$ the set
\[\Sigma\alpha:=\{\sigma\alpha:\sigma\in\Sigma\}\subseteq\A/\Q,\]
where $\sigma\alpha$ denotes the natural action of $\Q$ on $\A/\Q$, is dense in $\A/\Q$.
\end{theorem}
The paper is organized as follows. In Section \ref{sec.prelim} we give a basic background about the Adeles and a metric which generates the topology on $\A/\Q$. In Section \ref{sec.main proof} we give the proof of our main theorem, and in Section \ref{sec.non dense} we give several examples of non-dense orbits for the actions of semigroups of $\Q$ on $\A/\Q$. The reader who is new to these ideas is encouraged to read the final section before the proof of the main result.

Acknowledgements: We would like to thank Tom Ward and Manfred Einsiedler for helpful comments regarding an earlier draft of this paper.

\section{Preliminaries}\label{sec.prelim}
We use $\A$ to denote the rational Adeles with the usual restricted product topology, under which it is a locally compact Abelian group (we refer the reader to \cite[Chapter 5]{Rama1999} for a basic treatment of the Adeles). The group $\Q$ diagonally embeds in $\A$, and we also denote the image of this embedding as $\Q$. This slight ambiguity in notation should cause no confusion in what follows. It is well known that $\Q$ is a discrete subgroup of $\A$ and that the quotient $\A/\Q$ is compact. Our canonical choice of fundamental domain for this quotient will be
\[\F:=[0,1)\times\prod_p\Z_p.\]
We will view $\F$ as a topological space, with a base for its topology being given by the following two types of sets:
 \begin{enumerate}
 \item[(i)] Sets of the form \[U_\infty\times\prod_pU_p,\] where each $U_p\subseteq\Z_p$ is an open ball, $U_\infty\subseteq (0,1)$ is an open interval, and $U_p=\Z_p$ for all but finitely many primes.\vspace*{.1in}
 \item[(ii)] Sets of the form
 \[\left((\alpha,1)\times \prod_{p}U_p\right)\cup\left([0,\beta)\times \prod_{p}(U_p-1)\right),\]
 where $\alpha,\beta\in (0,1)$, each $U_p\subseteq\Z_p$ is an open ball, and $U_p=\Z_p$ for all but finitely many primes.
 \end{enumerate}
The reason for this choice is made clear by the following proposition.
\begin{prop}
As a topological space $\A/\Q$ is homeomorphic to $\F$.
\end{prop}
\begin{proof}
Call the two types of sets in the basis above sets of types (i) and (ii). First we show that both types of sets are open in the quotient topology on $\A/\Q$. Let $\phi:\A\rar\A/\Q$ be the quotient map, and suppose that $X\subseteq \F$ is a set of type (i). Then
\[\phi^{-1}\left(X\right)=\bigcup_{\eta\in\Q}(X+\eta),\]
and for any $\eta\in\Q$ we have that $U_p+\eta=\Z_p$ for all but finitely many primes. Therefore each set $X+\eta$ is open in $\A$ and we conclude that $X$ is open in $\A/\Q$.

Similarly, if $X\subseteq\F$ is a set of type (ii) then write $X=A\cup B$ with
\[A=\left((\alpha,1)\times \prod_{p}U_p\right)~\text{ and }~B=\left([0,\beta)\times \prod_{p}(U_p-1)\right),\]
and note that
\begin{align*}
\phi^{-1}\left(X\right)&=\bigcup_{\eta\in\Q}((A+\eta)\cup (B+(\eta+1))\\
&=\bigcup_{\eta\in\Q}\left(\left((\alpha,1+\beta)\times \prod_{p}U_p\right)+\eta\right)
\end{align*}
is open in $\A$.

Next we show that any open set in $\A/\Q$ can be written as a union of sets of types (i) and (ii). If $X\subseteq\A/\Q$ is open then we can write
\[\phi^{-1}(X)=\bigcup_{i\in I}Y_i,\]
where each $Y_i\subseteq\A$ is a set of the form
\[Y_i=V_{i,\infty}\times\prod_pV_{i,p},\]
with each $V_{i,p}\subseteq\Q_p$ an open ball, $V_{i,p}=\Z_p$ for all but finitely many primes, and $V_{i,\infty}\subseteq\R$ an open interval of length less than one. Then we have that
\[X=\bigcup_{i\in I}\phi (Y_i),\]
and each of the sets in this union, viewed as a subset of $\F$, is a set of type (i) or (ii).
\end{proof}
As a word of warning we point out that, although their topological structure is similar, $\A/\Q$ is \emph{not} isomorphic as a group to the direct product
\[\R/\Z\times\prod_{p}\Z_p.\]
This fact has already been featured in the proof above, and indeed the group structure is reflected by the shapes of the sets of type (ii).

For convenience we will work with a metric on $\A/\Q$ which induces its topology. We define $d:\F\times\F\rar [0,\infty)$ by
\begin{equation*}
d(\alpha,\beta)=\min_{\eta\in\{0,\pm 1\}}\max\left(|\alpha_\infty-\beta_\infty-\eta|_\infty,\max_p\left(\frac{|\alpha_p-\beta_p-\eta|_p}{p}\right)\right).
\end{equation*}
The function $d$ defines a metric on $\F$ and the metric topology is the same as the quotient topology on $\A/\Q$, since it is precisely that generated by the type (i) and (ii) sets above (see also \cite[Proposition 3.3]{TorbZuni2012}). We will use the following proposition to quantify the density of the integers at the finite places of $\A/\Q$.
\begin{prop}\label{prop.eps dense in A/Q}
Let $\varepsilon>0$ and suppose that $X\subseteq [0,1)$ is an $\varepsilon-$dense subset of $[0,1)$. For each prime $p$ let
\begin{equation}\label{eqn.k_p def}
k_p=k_p(\varepsilon):=\max\left(0,\lceil-\log_p\varepsilon\rceil-1\right),
\end{equation}
and set
\begin{equation}\label{eqn.N eps def}
N=N(\varepsilon):=\prod_{p}p^{k_p}.
\end{equation}
Then, with respect to the metric $d$, the set
\[\{(x;n,n,n,\ldots):x\in X, 1\le n\le N\}\subseteq\F\]
is an $\varepsilon-$dense subset of $\A/\Q$.
\end{prop}
\begin{proof}
The diameter of the space $\A/\Q$ is $1/2$, so we may assume that $\varepsilon<1/2$. For any $\alpha\in\F$ we can clearly choose $x\in X$ so that $|\alpha_\infty-x|_\infty\le \varepsilon$. Furthermore by the Chinese Remainder Theorem we can choose an integer $1\le n\le N$ so that
\[n=\alpha_p~\mathrm{mod}~ p^{k_p}\]
for all primes $p$ (note that $k_p=0$ for all primes $p\ge 1/\varepsilon$). Then by the definition of $d$ we have that
\begin{align*}
d(\alpha,(x;n,n,n,\ldots)) &= \max\left(|\alpha_\infty-x|_\infty~,~\max_p\frac{|\alpha_p-n|_p}{p}\right)\\
&\le \max\left(\varepsilon~,~\max_p p^{-k_p-1}\right).
\end{align*}
By the choice of $k_p$ this is less than or equal to $\varepsilon$.
\end{proof}

\section{Proof of Theorem \ref{thm.dense orbits}}\label{sec.main proof}
The outline of our proof of Theorem \ref{thm.dense orbits} is based on Furstenberg's original proof of his orbit closure theorem \cite{Furs1967} and also on Boshernitzan's exposition of this result \cite{Bosh1994}. However there are many features of our proof which are new, which do not appear when one is considering the group $\R/\Z$.

To begin, by replacing $\gamma$ with $\gamma^2$ and $\delta$ with $\delta^2$ if necessary, we assume without loss of generality that $\gamma>0$ and $\delta>1$. We first have the following proposition.
\begin{prop}\label{prop.close to 1}
For any $\varepsilon>0$ there exist $a\in\N$ and $b\in\Z$ with
\begin{equation}\label{eqn.sigma ineqs}
1<\frac{\gamma^a}{\delta^b}<1+\varepsilon.
\end{equation}
\end{prop}
\begin{proof}
Taking logarithms in (\ref{eqn.sigma ineqs}) and dividing by $b\log \gamma$, this is equivalent to the statement that, for any $\varepsilon'>0$, there are $a\in\N$ and $b\in\Z$ with
\[0<\frac{a}{b}-\frac{\log \delta}{\log \gamma}<\frac{\varepsilon'}{b}.\]
Now we can take $a/b=p_{2n+1}/q_{2n+1}$ to be the $(2n+1)^{st}$ convergent in the continued fraction expansion of $\log\delta/\log\gamma$, with $n$ chosen large enough that $q_{2n+1}^{-1}<\varepsilon'.$ If $\gamma<1$ then we choose $a>0, b<0$, while if $\gamma>1$ we choose $a,b>0$.
\end{proof}
Now let $\mathcal{P}$ be the collection of primes $p$ for which $|\gamma|_p>1$, define
\begin{equation}\label{eqn.P defn}
P=\prod_{p\in\mathcal{P}}|\gamma|_p,
\end{equation}
and write $P=p_1^{a_1}\cdots p_m^{a_m}$ with $p_1,\ldots , p_m$ distinct primes. Let
\[X_\mathcal{P}=\left.\left(\R\times\prod_{i=1}^m\Q_{p_i}\right)\right/\iota (\Z[1/(p_1\cdots p_m)]),\]
where
\[\iota:\Z[1/(p_1\cdots p_m)]\rightarrow\R\times\prod_{i=1}^m\Q_{p_i}\]
denotes the diagonal embedding, and let
\[\pi:\A/\Q\rightarrow X_\mathcal{P}\]
be the map obtained by projection of $\F$ onto $X_\mathcal{P}$. We take $X_\mathcal{P}$ with its usual quotient topology and note that the map $\pi$ is continuous. Furthermore the elements of $\Sigma$ act both on $\A/\Q$ and on $X_\mathcal{P}$ as endomorphisms, and these actions commute with $\pi$. We have the following proposition.
\begin{prop}\label{prop.rational accum pt 1}
If $\alpha\in\A/\Q$ satisfies $\alpha_\infty\not\in\Q$, and if there exists an $s/t\in\Q$ such that $\iota (s/t)$ is an accumulation point of the set $\pi(\Sigma \alpha)$, then $\Sigma\alpha$ is dense in $\A/\Q$.
\end{prop}
\begin{proof}
We first claim that we can replace $\iota (s/t)$ with an accumulation point satisfying the same hypotheses, which is also invariant under the action of suitable powers of $\gamma$ and $\delta$. If it happens that $s/t=0$ then this claim is trivially satisfied, so let us assume temporarily that $s/t\not=0.$

Write $\gamma=r/P$ with $(r,P)=1$. We can assume that $s/t$ lies in the fundamental domain
\begin{equation}\label{eqn.fund dom 2}
(-1,0]\times\prod_{p\in\mathcal{P}}\Z_p,
\end{equation}
and we can also assume (by multiplying by an element of $\Sigma$ if necessary), that $(r,t)=(\delta,t)=1.$ Here and throughout the proof we are using the fact, mentioned before the proposition, that
\[\pi (\sigma\alpha)=\sigma\pi (\alpha), ~\text{for any}~ \alpha\in\A/\Q ~\text{and}~ \sigma\in\Sigma.\]
We can view $s/t$ as an element of the ring
\[\lim_\leftarrow\Z/P^i\Z,\]
with an eventually periodic expansion. For each $i\in\N$ choose $0\le A_i<P^i$ with $s/t=A_i~\mathrm{mod}~P^i$. Then there is an $I\in\N$ such that for all $i>I$, the rational number
\begin{equation}\label{eqn.periodic expan}
\xi_i=P^{-i}\left(\frac{s}{t}-A_i\right)=\frac{s}{P^it}-\frac{A_i}{P^i}\in\Z_p
\end{equation}
has a purely periodic $p$-adic expansion, for all primes $p\in\mathcal{P}$. We may assume that $I$ is large enough so that $\xi_i\in (-1,0)$ for all $i>I$. Now choose $\mu >I$ such that $(r\delta)^\mu=1~\mathrm{mod}~t$, and let $\beta'=(r\delta)^\mu(s/t)$ and $\beta''=(\gamma\delta)^\mu(s/t)$. Then, choosing a representative in (\ref{eqn.fund dom 2}) for $\beta'$, we have that
\[\beta'_\infty=\frac{(r\delta)^\mu s}{t}-\left(1+\left\lfloor\frac{(r\delta)^\mu s}{t}\right\rfloor\right)=\frac{s}{t},\]
and therefore also that $\beta'_p=s/t$ for all $p\in\mathcal{P}.$ Since $\beta''=P^{-\mu}\beta'$, we then have that $\beta''_\infty=\xi_\mu$ and that $\beta''_p=\xi_\mu$ for all $p\in\mathcal{P}.$ Here our representative for $\beta''$ is determined uniquely by the requirements that $\beta''_\infty\in (-1,0]$ and $\beta''_p\in\Z_p$ for all $p$. Therefore, for the remainder of this proof we will assume (by replacing $s/t$ with $\xi_\mu$) that the hypotheses of this proposition are satisfied for a fraction $s/t\in (-1,0)$ with $(r\delta,t)=1$, with $(p,t)=1$ for all $p\in\mathcal{P}$, and with the property that the $p$-adic expansion of $s/t$ is purely periodic, for all $p\in\mathcal{P}.$

Let $\nu\in\N$ be the least common multiple of $\varphi (t)$ and all of the period lengths of $s/t$ modulo primes $p\in\mathcal{P}$. Set $\beta^{(0)}=r^\nu(s/t)$, and for each $1\le i\le m$ set $\beta^{(i)}=p_i^{-\nu a_i}\beta^{(i-1)}$. We will work with representatives for these elements which lie in the fundamental domain (\ref{eqn.fund dom 2}).

First of, as in the argument above, notice that
\[\beta_\infty^{(0)}=\frac{r^\nu s}{t}-\left(1+\left\lfloor\frac{r^\nu s}{t}\right\rfloor\right)=\frac{s}{t},\]
and therefore that $\beta^{(0)}_p=s/t$ for all $p\in\mathcal{P}$. Similarly for each $1\le i\le m$ we have that
\[\beta^{(i)}_{p_i}=\frac{s}{t},\]
because the $p_i$-adic expansion of $s/t$ is periodic with period dividing $\nu$. Therefore we have that $\beta^{(i)}_{\infty}=s/t$ and $\beta^{(i)}_{p}=s/t$ for all $p\in\mathcal{P}$ and for $1\le i\le m$. In particular, this is true for $i=m$, when $\beta^{(m)}=\gamma^\nu(s/t)$. The same argument applies with $\gamma$ replaced by $\delta$, and this verifies the claim which we made at the beginning of the proof.

Let us take $s/t$ as above, satisfying our claim, and for notational convenience let us replace $\gamma$ and $\delta$ by $\gamma^\nu$ and $\delta^\nu$, and $\Sigma$ by the semigroup generated by the new choices of $\gamma$ and $\delta$. The next part of the proof is somewhat technical, but the central idea is to find a point $y\in\A/\Q$ for which $\pi (y)$ is close to $\iota (s/t)$, and then multiply $y$ by a sequence of elements of $\Sigma$, selected with the aid of Proposition \ref{prop.close to 1}, to move away from $s/t$ in the Archimedean direction in very small steps. We continue on in small steps, controlling the $p$-adic directions so that they stay in $\Z_p$, until the Archimedean component of our point is close to an integer $N$ which is as large as we need. Then we translate everything back into a fundamental domain. The real components will remain in a sufficiently dense set in the real direction, and the $p$-adic components will simultaneously wind around enough in the $p$-adic directions to allow us to apply Proposition \ref{prop.eps dense in A/Q}.

Now we fill in the details of this argument. Let $M\ge 2$ be an even integer, for each prime $p$ let $k_p=k_p(M^{-1})$ be defined as in (\ref{eqn.k_p def}), and let $N=N(M^{-1})$ be defined as in (\ref{eqn.N eps def}). Then let
\[L=\mathrm{lcm}_{p\le M}\left(p^{k_p-1}(p-1)\right),\]
and apply Proposition \ref{prop.close to 1} to find integers $a\in\N$ and $b\in\Z$ with
\[1<\frac{\gamma^a}{\delta^b}<1+\frac{1}{LMN}.\]
Next set $\sigma=\gamma^a/\delta^b$ and let
\[K=\left\lceil 2\delta MN(\sigma-1)^{-1}\sigma^L\right\rceil.\]
If $b<0$ then choose $\gamma_0\in\Sigma$ to be any element satisfying $\gamma_0>1,$ and if $b>0$ choose $\gamma_0=\delta^{Kb}$. Then set
\[\gamma_i=\sigma^i\gamma_0~\text{ for }~ 1\le i\le K,\]
and note that each of these fractions is in $\Sigma$. Finally choose a point $y\in\Sigma\alpha$ in the fundamental domain
\begin{equation}\label{eqn.fund dom 3}
(-1,0]\times\prod_p\Z_p
\end{equation} satisfying
\begin{align}
\left| y_\infty-\frac{s}{t}\right|_\infty &<\frac{1}{2\gamma_0\sigma^LM},\quad\text{and}\label{eqn.y_inf upp bnd}\\
\left| y_p-\frac{s}{t}\right|_p&<\frac{\min\left(1,|\sigma|^{-K}_p\right)}{p^{k_p}}\quad\text{for all}\quad p\in\mathcal{P}.\label{eqn.y_p upp bnd}
\end{align}
By multiplying $y$ by a suitable power of $\delta$, which does not increase the $p$-adic absolute values of any of the $y_p$-coordinates, we may assume that we have chosen $y$ so that $y_\infty$ also satisfies the lower bound
\begin{equation}\label{eqn.y_inf low bnd}
\left| y_\infty-\frac{s}{t}\right|_\infty>\frac{1}{2\delta\gamma_0\sigma^LM}.
\end{equation}
We point out that this is where we are using the fact that $\alpha_\infty\not\in\Q$. Now for each $0\le i\le K$ define $y^{(i)}=\gamma_iy.$ If we write $y^{(0)}_\infty=s/t+x_\infty$ and $y^{(0)}_p=s/t+x_p$ for each $p\in\mathcal{P}$, then we have for each $0\le i\le K$ that
\begin{align*}
y^{(i)}_\infty&=\sigma^i\cdot\frac{s}{t}+\sigma^ix_\infty,\\
y^{(i)}_p&=\sigma^i\cdot\frac{s}{t}+\sigma^ix_p\quad\text{for}\quad p\in\mathcal{P},\quad\text{and}\\
y^{(i)}_p&=\sigma^iy^{(0)}_p\quad\text{for}\quad p\not\in\mathcal{P}.
\end{align*}
By our choice of $s/t$ there is a fraction $s'/t'$ which is in $\Z_p$ for all $p\not\in\mathcal{P}$ and which satisfies
\[\sigma\cdot\frac{s}{t}=\frac{s}{t}+\frac{s'}{t'}.\]
Noting that inequality (\ref{eqn.y_p upp bnd}) guarantees that $|\sigma^i x_p|_p<p^{-k_p}$ for all $p\in\mathcal{P}$ and $0\le i\le K$, we have after translating back into the fundamental domain (\ref{eqn.fund dom 3}) that
\begin{align*}
y^{(i)}_\infty&=\frac{s}{t}+\sigma^ix_\infty-n_i,\\
y^{(i)}_p&=\frac{s}{t}+\sigma^ix_p-n_i\quad\text{for}\quad p\in\mathcal{P},\quad\text{and}\\
y^{(i)}_p&=\sigma^iy^{(0)}_p-\frac{s'}{t'}\cdot\sum_{j=0}^{i-1}\sigma^j-n_i\quad\text{for}\quad p\not\in\mathcal{P},
\end{align*}
where
\[n_i=1+\left\lfloor\frac{s}{t}+\sigma^ix_\infty\right\rfloor.\]
Now by our lower bound (\ref{eqn.y_inf low bnd}) we have for any $0\le i< K$ that
\[|\sigma_{i+1}x_\infty-\sigma_ix_\infty|_\infty=|(\sigma-1)\sigma^ix_\infty|_\infty>\frac{(\sigma-1)}{2\delta\sigma^LM},\]
and by our choice for $K$ this implies that $\sigma^Kx_\infty>N$. On the other hand for any $i$ with $\sigma^i< N$ we have that
\[|\sigma_{i+1}x_\infty-\sigma_ix_\infty|_\infty<\frac{1}{LM}.\]
This together with our upper bound (\ref{eqn.y_inf upp bnd}) ensures that for any interval $\mathcal{I}\subseteq [0,N]$ of length $2/M$, there is an integer $0\le j\le K-L$ with the property that $\sigma^ix_\infty\in \mathcal{I}$ for all $j\le i\le j+L$. One of these integers must equal $0$ modulo $L$ and for this integer we have that
\begin{align*}
\sigma^i=1~\mathrm{mod}~p^{k_p}~\text{ for all}~p\not\in\mathcal{P},
\end{align*}
and therefore also that
\[\sum_{j=0}^{i-1}\sigma^j=\frac{\sigma^i-1}{\sigma-1}=0~\mathrm{mod}~p^{k_p}~\text{ for all}~p\not\in\mathcal{P}.\]

To finish the argument, partition the interval $(-1,0]$ into disjoint intervals of length $2/M$. By what we have said above, for each such interval $\mathcal{I}$ and for each integer $1\le n\le N$ there is an $0\le i\le K$ with $y_\infty^{(i)}\in\mathcal{I}$ and $|y_p^{(i)}-n|_p\le p^{-k_p}$ for all primes $p$. By Proposition \ref{prop.eps dense in A/Q} the set $\Sigma\alpha$ is $2/M$-dense in $\A/\Q$, and $M$ can be taken arbitrary large.
\end{proof}
Now we come to the proof of our main result.
\begin{proof}[Proof of Theorem \ref{thm.dense orbits}] We will show that if $\alpha\in\A/\Q$ satisfies $\alpha_\infty\not\in\Q$, then there is a rational $s/t$ satisfying the hypothesis of Proposition \ref{prop.rational accum pt 1}. Suppose, by way of contradiction, that this is not the case. As before write $\gamma=r/P$ with $(r,P)=1$. Then for any integer $v\ge 3$ which is coprime to $Pr\delta$, there exists an integer $u$ with $-v<u\le -1$ and $(u,v)=1$, for which $u/v$ has a purely $p$-adic expansion for all $p\in\mathcal{P}$. This follows from the same argument used to construct the fractions $\xi_i$ in the previous proof (cf. (\ref{eqn.periodic expan})). By replacing $\gamma$ and $\delta$ by suitable powers, as in the previous proof, we may assume that $\iota (u/v)$ is invariant under the action of $\Sigma$.

Define $Y_0\subseteq X_\mathcal{P}$ by
\[Y_0=\overline{\pi\left(\Sigma\alpha\right)},\]
and for each $i\in\N$ define $Y_i\subseteq Y_{i-1}$ by
\[Y_i=\{y\in Y_{i-1}:y+\iota (u/v)\in Y_{i-1}\}.\]
For any $i\in\N$ any for any  $y\in Y_i$ we have that
\[y+j\cdot\iota (u/v)\in Y_0 ~\text{ for all}~0\le j\le i.\]
Therefore if we can show that all of the sets $Y_i$ are non-empty then, since $i$ can be taken arbitrarily large, it will follow from Proposition \ref{eqn.N eps def} that the set $Y_0$ must be $1/v$-dense in $X_\mathcal{P}$. Since $Y_0$ is closed and $v$ can be taken arbitrarily large, this would mean that $Y_0=X_\mathcal{P}$, which would contradict our initial assumption and complete the proof.

All that remains is to show that the sets $Y_i$ are always non-empty, and we will do this by an inductive argument. It is clear that $Y_0$ is non-empty, closed, and $\Sigma$-invariant. Since $\alpha_\infty\not\in\Q$, it follows that $Y_0$ contains an accumulation point $y$ and a sequence of points $(y^{(j)})_{j=1}^\infty$ approaching $y$, with the property that $y^{(j)}_\infty\not=y_\infty$ for any $j$. This implies that there is a sequence of points $(x^{(j)})_{j=1}^\infty$ in the difference set
\[D_0:=Y_0-Y_0\]
with $x^{(j)}\rar 0$ and $x^{(j)}_\infty\not=0$ for any $j$. Since $D_0$ is closed, this allows us to apply the same argument as in the previous proposition to conclude that $D_0=X_\mathcal{P}$. We note that the property that $x^{(j)}_\infty\not=0$ for any $j$ is crucial here (cf. (\ref{eqn.y_inf low bnd})). Since $u/v\in D_0$ we conclude that $Y_1$ is non-empty.

Next we show that $Y_2$ is non-empty. The set $Y_1$ is non-empty, closed, and $\Sigma$-invariant (since $\iota(u/v)$ is $\Sigma$-invariant). If there is an element $y\in Y_1$ with $y_\infty\not\in\Q$ then the argument is exactly the same as before. If not, then choose any element $y\in Y_1$ and, considering it as an element of the fundamental domain (\ref{eqn.fund dom 2}), choose a prime $p\in\mathcal {P}$ with the property that $y_p\not=y_\infty$. This is possible since we know that $\iota (y_\infty)\not\in Y_0$. Then $y_p\in\Z_p$ and $y_\infty\in\Q\cap\Q_p$, and we can write their $p$-adic expansions as
\begin{align*}
y_\infty=\sum_{\ell=-\infty}^\infty b_\ell p^\ell\quad\text{and}\quad y_p=\sum_{\ell=0}^\infty c_\ell p^\ell,
\end{align*}
with $b_\ell=0$ for all $\ell$ less than some bound. Now write $\gamma=p^{-a}\gamma'$ with $|\gamma'|_p=1$, and set $z^{(i)}=p^{-ia}y$ and $y^{(i)}=\gamma^iy$. Working in the fundamental domain (\ref{eqn.fund dom 2}), we find that for all sufficiently large $i$,
\begin{align*}
|z^{(i)}_\infty|_p=p^{ia}\cdot\left|\sum_{\ell=-\infty}^{-1}b_{\ell}p^\ell+\sum_{\ell=0}^{ia-1}(b_{\ell}-c_{\ell})p^\ell\right|_p.
\end{align*}
Here we are using the fact that $b_\ell\not= c_\ell$ for some $\ell$, and this also implies that $|z^{(i)}_\infty|_p$ tends to infinity with $i$. Since $|\gamma'|_p=1$, and since translation by a $p$-adic integer (in order to bring all other coordinates back into the fundamental domain) does not change the $p$-adic absolute value of an element of $\Q_p\setminus\Z_p$, we have that $|y^{(i)}_\infty|_p$ tends to infinity with $i$. This means that there are infinitely many points in $Y_1$ whose representatives in the fundamental domain (\ref{eqn.fund dom 2}) have different Archimedean coordinates, and therefore there is a sequence of points $(x^{(j)})_{j=1}^\infty$ in the difference set $D_1:=Y_1-Y_1$ with $x^{(j)}\rar 0$ and $x^{(j)}_\infty\not=0$ for any $j$. The rest of the argument for showing that $Y_2$ is non-empty is exactly the same as before, and the same argument then shows that $Y_i$ is non-empty for all $i\in\N.$ This concludes the proof of Theorem \ref{thm.dense orbits}.
\end{proof}

\section{Examples of non-dense orbits}\label{sec.non dense}
Finally we give examples which illustrate various ways in which one can fail to have dense orbits. Firstly, if all the elements of $\Sigma$ are contractions on $\R$, then there is one obvious degeneracy.
\begin{prop}\label{prop.non-dense orbit 1}
Suppose that $\gamma,\delta\in\Q$ satisfy
\[\max\left(|\gamma|,|\delta|\right)<1,\]
let $\Sigma$ denote the multiplicative semigroup which they generate, and let $P$ be defined as in (\ref{eqn.P defn}). If $\alpha\in\A/\Q$ is any point with coordinates in the fundamental domain
\begin{equation}\label{eqn.fund dom 4}
[-1/2,1/2)\times\prod_p\Z_p
\end{equation}
which satisfies $\alpha_p=0$ for all $p~|P$, then all accumulation points $x$ of the set $\Sigma\alpha$ have $x_\infty=0$.
\end{prop}
\begin{proof}
If $\alpha$ is any point which satisfies the hypotheses of this proposition then for any $\sigma\in\Sigma$ and for any prime $p$, we will have that $\sigma\alpha_p\in\Z_p$. However as $a$ and $b$ tend to infinity along any sequence in $\N$, we will have that
\[\gamma^a\delta^b\alpha_\infty\rightarrow 0,\]
so that $0$ is the only possible real coordinate of any accumulation point of $\Sigma\alpha$ in the fundamental domain (\ref{eqn.fund dom 4}).
\end{proof}
Secondly, there are some semigroups of contractions which give rise to orbits whose real coordinates become dense in fractal sets.
\begin{prop}
Let $\Sigma$ be the multiplicative semigroup generated by $1/4$ and $5/64$. There are points $\alpha\in\A/\Q$ with the property that
\[\overline{\left\{(\sigma\alpha)_\infty\in [0,1):\sigma\in\Sigma\right\}}\]
is a set of Hausdorff dimension equal to $1/2$.
\end{prop}
\begin{proof}
Define affine contractions $T_1,T_2,T_3,T_4:[0,1)\rar [0,1)$ by setting
\begin{align*}
T_1(x):=x/4+1/4,\quad &T_2(x):=x/4+1/2\\
T_3(x):=5x/64+41/64,\quad&\text{and}\quad T_4(x):=5x/64+9/32.
\end{align*}
Then there exists a unique non-empty compact set $F$ such that $F=\bigcup_{i=1}^4 T_i(F)$ (see Theorem 9.1 in \cite{Falc1990}). Noting that $T_3([0, 1))\subset T_2([0, 1))$ and $T_4([0, 1))\subset T_1([0, 1))$, we may use Hutchinson's Formula (see Theorem 9.3 in \cite{Falc1990}) to calculate the Hausdorff dimension of the attractor $F$. We have that $\mathrm{dim}_{\mathrm{H}}(F)=s$, where $4^{-s}+4^{-s}=1$, that is, $s=1/2$.

Let $x\in [0,1)$ be a point whose orbit under the collection of maps $T_i$ is dense in $F$, and choose $\alpha\in\A/\Q$ to be any point in $\F$ with $\alpha_\infty=x$ and $\alpha_2=1/3$. The action of multiplying $\alpha$ (or any point in $\F$ with $2$-adic coordinate $1/3$) by $1/4$, and then translating back to $\F$, corresponds to the map $T_1$, extended coordinate-wise to $\A/\Q$. Also, the $2$-adic coordinate of the image is $1/3$.

Similarly, multiplication of any point in $\F$ with $2$-adic coordinate $1/3$ by $5/64$ corresponds to the map $T_3$, and the $2$-adic coordinate of the image is $2/3$. The same argument applies to points with $2$-adic coordinates $2/3$, and the resulting maps correspond to $T_2$ and $T_4$. Therefore the real part of the closure of the orbit of $\alpha$ under the semigroup $\Sigma$ equals $F$.
\end{proof}
Finally, if $\Sigma$ is generated by finitely many integers, then there will be other types of infinite non-dense orbits (of points $\alpha$ with $\alpha_\infty\in\Q$). For example the orbit of any point $\alpha\in\A/\Q$ with $\alpha_\infty=2/7$ and $\alpha_p\not=2/7$ for any $p$ under the action of the semigroup generated by $8$ and $729$ will be infinite, but all points in the closure will have real coordinate $2/7$. We point this out only in order to emphasize that the correct interpretation must be applied when using Berend's hypotheses, for example in \cite[Theorems II.1, IV.1]{Bere1984}.


\begin{thebibliography}{1}


\bibitem{Bere1984}
D.~Berend: \emph{Multi-invariant sets on compact abelian groups},
 Trans. Amer. Math. Soc.  286  (1984),  no. 2, 505-535.

\vspace*{1ex}


\bibitem{Bosh1994}
M.~D.~Boshernitzan: \emph{Elementary proof of Furstenberg's Diophantine result},
    Proc. Amer. Math. Soc. 122 (1994), no. 1, 67-70.

\vspace*{1ex}

\bibitem{Falc1990} K.~J.~Falconer: \emph{Fractal Geometry}, Wiley and Sons, Chichester, 1990.

\vspace*{1ex}

\bibitem{Furs1967}
H.~Furstenberg: \emph{Disjointness in ergodic theory, minimal sets, and a problem in
 Diophantine approximation}, Math. Systems Theory 1 (1967), 1-49.


\vspace*{1ex}

\bibitem{Rama1999}
D.~Ramakrishnan, R.~J.~Valenza: \emph{Fourier analysis on number fields}, Graduate Texts in Mathematics, 186. Springer-Verlag, New York, 1999.

\vspace*{1ex}

\bibitem{TorbZuni2012}
S.~M.~Torba and W.~A.~Z\'{u}\~{n}iga-Galindo: \emph{Parabolic type equations and Markov stochastic processes on Adeles}, arxiv.org/abs/1206.5213.



\end{thebibliography}
\end{document}